\documentclass[a4paper,12pt]{amsart}
\usepackage{amsfonts}
\usepackage{amssymb}
\usepackage{ifthen}
\usepackage{graphicx}
\nonstopmode \numberwithin{equation}{section}
\usepackage[usenames]{color}

\newtheorem{theorem}[equation]{Theorem}
\newtheorem{lemma}[equation]{Lemma}
\newtheorem{corollary}[equation]{Corollary}

\numberwithin{equation}{section}
\newtheorem{case}[equation]{Case}
\newtheorem{subcase}[equation]{Subcase}
\newtheorem{claim}[equation]{Claim}
\theoremstyle{definition}
\newtheorem{definition}[equation]{Definition}
\newtheorem{remark}[equation]{Remark}

\newtheorem{examples}[equation]{Examples}

\newtheorem{thm}{Theorem}[section]
\newtheorem{lem}{Lemma}[section]
\newtheorem{cor}{Corollary}[section]

\newtheorem{cl}{Claim}[section]
\newtheorem{ca}{Case}[section]
\newtheorem{sca}{Subcase}[section]
\newtheorem{scl}[section]{Subclaim}
\newtheorem{conj}[equation]{Conjecture}

\theoremstyle{definition}
\newtheorem{defn}{Definition}[section]
\newtheorem{op}[equation]{Open Problem}
\newtheorem{ques}[equation]{Question}
\newtheorem{exam}[equation]{Example}

\newcounter {own}
\def\theown {\thesection       .\arabic{own}}

\newenvironment{pf}[1][]{%
 \vskip 3mm
 \noindent
 \ifthenelse{\equal{#1}{}}%
  {{\slshape Proof. }}%
  {{\slshape #1.} }%
 }%
{\qed\bigskip}

\newcounter{alphabet}
\newcounter{tmp}

\makeatletter
\newcommand{\Ref}[1]{\@ifundefined{r@#1}{}{\setcounter{tmp}{\ref{#1}}\Alph{tmp}}}
\makeatother

\newcommand{\diam}{{\operatorname{diam}}}

\def\be{\begin{equation}}
\def\ee{\end{equation}}

\newcommand{\bee}{\begin{enumerate}}
\newcommand{\eee}{\end{enumerate}}

\newcommand{\blem}{\begin{lem}}
\newcommand{\elem}{\end{lem}}
\newcommand{\bthm}{\begin{thm}}
\newcommand{\ethm}{\end{thm}}
\newcommand{\bcor}{\begin{cor}}
\newcommand{\ecor}{\end{cor}}
\newcommand{\beg}{\begin{exam}}
\newcommand{\eeg}{\end{exam}}
\newcommand{\begs}{\begin{examples}}
\newcommand{\eegs}{\end{examples}}
\newcommand{\bdefe}{\begin{defn}}
\newcommand{\edefe}{\end{defn}}
\newcommand{\bprob}{\begin{prob}}
\newcommand{\eprob}{\end{prob}}
\newcommand{\bques}{\begin{ques}}
\newcommand{\eques}{\end{ques}}
\newcommand{\bei}{\begin{itemize}}
\newcommand{\eei}{\end{itemize}}
\newcommand{\bcon}{\begin{conj}}
\newcommand{\econ}{\end{conj}}
\newcommand{\bop}{\begin{op}}
\newcommand{\eop}{\end{op}}

\newcommand{\bca}{\begin{ca}}
\newcommand{\eca}{\end{ca}}
\newcommand{\bsca}{\begin{sca}}
\newcommand{\esca}{\end{sca}}

\newcommand{\bcl}{\begin{cl}}
\newcommand{\ecl}{\end{cl}}

\newcommand{\bscl}{\begin{scl}}
\newcommand{\escl}{\end{scl}}

\newcommand{\bcons}{\begin{conjs}}
\newcommand{\econs}{\end{conjs}}
\newcommand{\bprop}{\begin{propo}}
\newcommand{\eprop}{\end{propo}}
\newcommand{\er}{\end{rem}}
\newcommand{\brs}{\begin{rems}}
\newcommand{\ers}{\end{rems}}
\newcommand{\bo}{\begin{obser}}
\newcommand{\eo}{\end{obser}}
\newcommand{\bos}{\begin{obsers}}
\newcommand{\eos}{\end{obsers}}
\newcommand{\bpf}{\begin{pf}}
\newcommand{\epf}{\end{pf}}
\newcommand{\ba}{\begin{array}}
\newcommand{\ea}{\end{array}}
\newcommand{\beq}{\begin{eqnarray}}
\newcommand{\beqq}{\begin{eqnarray*}}
\newcommand{\eeq}{\end{eqnarray}}
\newcommand{\eeqq}{\end{eqnarray*}}

\newcounter{minutes}
\divide\time by 60
\newcounter{hours}
\multiply\time by 60 \addtocounter{minutes}{-\time}

\begin{document}

\bibliographystyle{amsplain}

\title[On John domains in Banach spaces]
{On John domains in Banach spaces}
\author[]{Y. Li}
\address{Yaxiang. Li, College of Science,
Central South University of
Forestry and Technology, Changsha,  Hunan 410004, People's Republic
of China}
 \email{yaxiangli@163.com}
\author[]{M. Vuorinen}
\address{Matti. Vuorinen, Department of Mathematics and Statistics, University of Turku,
FIN-20014 Turku, Finland}
\email{vuorinen@utu.fi}

\author{X. Wang $^* $
}
\address{Xiantao. Wang, Department of Mathematics,
Hunan Normal University, Changsha,  Hunan 410081, People's Republic
of China} \email{xtwang@hunnu.edu.cn}

\date{}
\subjclass[2010]{Primary: 30C65, 30F45; Secondary: 30C20}
\keywords{removability, John domain, inner uniform domain, quasihyperbolic metric.\\
${}^{\mathbf{*}}$ Corresponding author\\
This research was supported by the Academy of Finland, Project 2600066611, and by NSF of
China (No. 11071063).}



\def\thefootnote{}
\footnotetext{ \texttt{\tiny File:~\jobname .tex,
          printed: \number\year-\number\month-\number\day,
          \thehours.\ifnum\theminutes<10{0}\fi\theminutes }
} \makeatletter\def\thefootnote{\@arabic\c@footnote}\makeatother

\begin{abstract}
We study the stability of John domains in Banach spaces under removal of a countable set of points.
In particular, we prove that the class of John domains is stable
in the sense that removing a certain type of closed countable set from
the domain yields a new domain which also is a John domain. We apply this result to prove the stability of the inner uniform domains. Finally,
we consider a wider class of domains, so called $\psi$-John domains
and prove a similar result for this class.
\end{abstract}

\maketitle\pagestyle{myheadings} \markboth{}{On John domains in Banach spaces}

\section{Introduction}
The class of domains, nowadays known as John domains and originally introduced by John \cite{J} in the study of elasticity theory, has been investigated during the past three decades by many people
in connection with applications of classical analysis and geometric
function theory. See for instance \cite{Bro, MS, NV} and the references
therein.
Here we study the classes of John domains and the wider class of $\psi$-John domains \cite{HV, Vai4} and the stability
of these two classes of domains under the removal of a countable closed set
of points. The motivation for this paper stems from the discussions in \cite{HPW,  Vai6}, where the effect of the removal of a finite set of points was examined.
See also the very recent paper \cite{klsv}.

Suppose that $D$ is a domain in a real Banach space $E$ with dimension at least $2$ and let  $P_D$ denote a countable set in $D$ such that the quasihyperbolic distance w.r.t. $D$ between each pair of distinct points in $P_D$ is at least $b$ where $b>0$ is a constant.
The first main result of this paper shows that
$D$ is a $c$-John domain if and only if $D\setminus P_D$ is a $c_1$-John
domain, where $c$ and $c_1$ are two constants depending only on each other and on $b$.
Applying this result, we show that $D$ is inner uniform if and only if $D\setminus P_D$ is inner uniform.  Our second main result shows that $D$ is a $\psi$-John domain if and only if $D\setminus P_D$ is a $\psi_1$-John
domain, where $\psi$ and $\psi_1$ depend only on each other and on $b$.

The methods applied in the proofs rely on standard
notions of metric space theory: curves, their lengths, and nearly length-minimizing curves. It
should be noted that we employ several metric space
structures on the domain $D$ including hyperbolic type metrics. We use three metrics: the norm metric, the distance ratio metric and the quasihyperbolic metric on the domain $D$ and, moreover, also on its subdomains.


\section{The second section}

Throughout the paper, we always assume that $E$ denotes a real
Banach space with dimension at least $2$. The norm of a vector $z$
in $E$ is written as $|z|$, and for each pair of points $z_1$, $z_2$
in $E$, the distance between them is denoted by $|z_1-z_2|$, the
closed line segment with endpoints $z_1$ and $z_2$ by $[z_1, z_2]$.
We always use $\mathbb{B}(x_0,r)$ to denote the open ball $\{x\in
E:\,|x-x_0|<r\}$ centered at $x_0$ with radius $r>0$. Similarly, for
the closed balls and spheres, we use the usual notations
$\overline{\mathbb{B}}(x_0,r)$ and $ \mathbb{S}(x_0,r)$,
respectively.

For each pair of points $z_1$, $z_2$ in $D$, the {\it distance ratio
metric} $j_D(z_1,z_2)$ between $z_1$ and $z_2$ is defined by
$$j_D(z_1,z_2)=\log\Big(1+\frac{|z_1-z_2|}{\min\{d_D(z_1),d_D(z_2)\}}\Big),$$where $d_D(z)$ denotes the
distance from $z$ to the boundary $\partial D$ of $D$.

The {\it quasihyperbolic length} of a rectifiable arc or a path
$\gamma$ in $D$ is the number (cf.
 \cite{Avv,GP,Geo,Vai6-0})
$$\ell_k(\gamma)=\int_{\gamma}\frac{1}{d_D(z)}\,|dz|.
$$

For each pair of points $z_1$, $z_2$ in $D$, the {\it
quasihyperbolic distance} $k_D(z_1,z_2)$ between $z_1$ and $z_2$ is
defined in the usual way:
$$k_D(z_1,z_2)=\inf\ell_k(\alpha),
$$
where the infimum is taken over all rectifiable arcs $\alpha$
joining $z_1$ to $z_2$ in $D$.

 For all $z_1$, $z_2$ in $D$, we have
(cf. \cite{Vai6-0})

\beq\label{eq(0000)} k_{D}(z_1, z_2)\geq
\inf\left\{\log\Big(1+\frac{\ell(\alpha)}{\min\{d_{D}(z_1), d_{D}(z_2)\}}\Big)\right\}\geq j_D(z_1, z_2)\eeq
$$ \geq
\Big|\log \frac{d_{D}(z_2)}{d_{D}(z_1)}\Big|,$$ where the infimum is
taken over all rectifiable curves $\alpha$ in $D$ connecting $z_1$
and $z_2$, $\ell(\alpha)$ denotes the length of $\alpha$. Next, if
$|z_1-z_2|< d_D(z_1)$, then we have \cite{Vai6-0}, \cite[Lemma
2.11]{vu81}
\begin{equation} \label{upperbdk}
k_D(z_1,z_2)\le \log\Big( 1+ \frac{
|z_1-z_2|}{d_D(z_1)-|z_1-z_2|}\Big) \le \frac{
|z_1-z_2|}{d_D(z_1)-|z_1-z_2|}\,,
\end{equation}
where the last inequality follows from the following elementary
inequality
$$
\frac{r}{1-r/2} \le \log \frac{1}{1-r} \le \frac{r}{1-r} \, \quad
{\rm for }\,\, 0\le r<1 \,.$$


Gehring and Palka \cite{GP} introduced the quasihyperbolic metric of
a domain in $R^n$. Many of the basic properties of this metric may
be found in \cite{Geo,HPWW,k,krt,rt,Vai9}. Recall that an arc $\alpha$ from $z_1$ to
$z_2$ is a {\it quasihyperbolic geodesic} if
$\ell_k(\alpha)=k_D(z_1,z_2)$. Each subarc of a quasihyperbolic
geodesic is obviously a quasihyperbolic geodesic. It is known that a
quasihyperbolic geodesic between every pair of points in $E$ exists
if the dimension of $E$ is finite, see \cite[Lemma 1]{Geo}. This is
not true in infinite dimensional Banach spaces (cf. \cite[Example 2.9]{Vai6-0}). In
order to remedy this shortage, V\"ais\"al\"a introduced the
following concepts \cite{Vai6}.

\begin{definition} \label{def1.4}Let $D\neq E$ and $c\geq 1$. An arc
$\alpha\subset D$ is a $c$-neargeodesic if
$\ell_k(\alpha[x,y])\leq c\;k_D(x,y)$ for all $x, y\in \alpha$.
\end{definition}

In \cite{Vai6}, V\"ais\"al\"a proved the following property
concerning the existence of neargeodesics in $E$.

\begin{lemma} $($\cite[Theorem 3.3]{Vai6}$)$ \label{ThmA}
Let $z_1,\, z_2 \in D\neq E$ and let $c>1$. Then there is a
$c$-neargeodesic  joining $z_1$ and $z_2$ in $D$.\end{lemma}

\begin{definition}\label{def-1} A domain $D$ in $E$ is called $c$-{\it John domain}
in the norm metric provided there exists a constant $c$ with the
property that each pair of points $z_{1},z_{2}$ in $D$ can be joined
by a rectifiable arc $\alpha$ in $ D$ such that for all $z\in \alpha$ the following holds:
\begin{equation} \label{eq-1} \min\{\ell (\alpha [z_1, z]), \; \ell (\alpha
[z_2, z])\}\leq c\,d_{D}(z),\end{equation}
 where $\alpha[z_{j},z]$ denotes the part of $\alpha$ between $z_{j}$ and $z$ (cf.
\cite{Bro, NV}). The arc $\alpha$ is called to be a {\it $c$-cone arc} .\end{definition}

A domain $D$ in $E$  is said to be a {\it $c$-uniform
domain} (cf. \cite{Martio-80,MS, Vai, Vai6,Vai5}) if there
is a constant $c\geq 1$ such that each pair of points $z_1,z_2\in D$
can be joined by an arc $\alpha$  satisfying \eqref{eq-1} and
\begin{equation}\label{eq-2} \ell(\alpha)\leq c\,|z_{1}-z_{2}|. \end{equation} We also say that $\alpha$ is a {\it  $c$-uniform
arc} (cf.
\cite{Vai4}).

For $z_1$, $z_2\in D$,  the {\it inner length metric $\lambda_D(z_1,
z_2)$} between these points is defined by
$$\lambda_D(z_1,z_2)=\inf \{\ell(\alpha):\; \alpha\subset D\;
\mbox{is a rectifiable arc joining}\; z_1\; \mbox{and}\; z_2 \}.$$

We say that a domain $D$ in $E$ is {\it an inner $c$-uniform domain} if there
is a constant $c\geq 1$ such that each pair of points $z_1,z_2\in D$
can be joined by an arc $\alpha$  satisfying \eqref{eq-1} and
\begin{equation}\label{eq-3} \ell(\alpha)\leq c \lambda_D(z_1, z_2).\end{equation}
Such an arc
 $\alpha$ is called to be an {\it inner $c$-uniform arc} (cf. \cite{Vai4}).

Obviously, uniform domains are inner uniform domains, but inner uniform does not imply uniform. See
\cite{BHK,FW,Geo,Martio-80,MS,  Vai,Vai6,Vai4} for more details on
uniform domains and inner uniform domains.

{\em Remarks.} If we replace \eqref{eq-1}, \eqref{eq-2} and \eqref{eq-3} by
\begin{equation} \label{eq-1'} \min\{\diam (\alpha [z_1, z]), \; \diam (\alpha
[z_2, z])\}\leq c\,d_{D}(z),\end{equation}
\begin{equation}\label{eq-2'} \diam(\alpha)\leq c\,|z_{1}-z_{2}| \end{equation} and
\begin{equation}\label{eq-3'} \diam(\alpha)\leq c \lambda_D(z_1, z_2),\end{equation}
then we get concepts which in the case $E=\mathbb{R}^n$ are $n$-quantitatively equivalent to $c$-John domain, $c$-uniform domain and inner $c$-uniform domain,
respectively \cite{NV}. But in an arbitrary Banach space,
each of these three conditions leads to a wider class of domains.
For example, the broken tube domain considered by V\"ais\"al\"a \cite[4.12]{Vai6-0} (see also \cite{Vai7})
is neither John nor quasiconvex (a metric space is {\it $c$-quasiconvex} if each pair of points $x,y$ can be joined by an arc with \eqref{eq-2} holds). Nevertheless, one can join a given pair of points in this bounded domain by arcs satisfying \eqref{eq-1'}, \eqref{eq-2'} and \eqref{eq-3'}.

Various classes of domains have been studied in analysis (e.g. see \cite{HPS}).
For some classes, the removal of a finite number of points from a domain may yield
a domain no longer in this class \cite{HPS}. In \cite{HPW}, the authors proved that the removal of a finite number of points from a
John domain yields another John domain.

\begin{lemma}\label{ThmB} $($\cite[Theorem 1.4]{HPW}$)$~~A domain $D\subset \mathbb{R}^n$ is a John domain if and only if
$G=D\setminus P$ is also a John domain, where
$P=\{p_1,p_2,\cdots,p_m\}\subset D.$\end{lemma}

In general, when $P$ is an infinite closed set in $D$, $D\setminus P$ need
not be a John domain $($\cite[Example 1.5]{HPW}$)$. In this paper, we continue the study of the removability properties of John domains and prove that if $P$
satisfies a certain separation condition, then $D\setminus P$ is still a John domain
if $D$ is a John domain.

Let $b>0$ be a constant. In what follows, for a domain $D$ in $E$, and for
a sequence $X= \{x_j: j=1,2,...\}$ of points in $D$
satisfying {\it the quasihyperbolic separation condition}
$$  k_D(x_i,x_j)\geq b\;\; {\rm for}\;\; i\neq j, $$
 we always write
$$P_D=X.$$
Further, we assume that the set $P_D$ satisfying the quasihyperbolic separation condition contains at least two points, and in the following, without loss of generality, we may assume that $b=\frac{1}{2}$.

Given $x\in D$ and $s\in(0,1),$ for $z_1, z_2\in \mathbb{B}(x,s
d_D(x))$, we see from \eqref{upperbdk} that
$$k_D(z_1,z_2)\leq 2 \log (1/(1-s)).$$ This fact, together with the definition of $P_D,$ yields the following lemma.

\begin{lemma}\label{star}For all $w\in D$, there exists at most one point $x_i$ of $P_D$ such that $x_i\in \mathbb{B}(w,\frac{1}{6}d_D(w))$.\end{lemma}

\begin{lemma} \label{ThmC}{\rm (\cite[Lemma 6.7]{Vai6})} Suppose that $G$ is a $c$-uniform domain and that $x_0\in G$. Then $G_0=G\setminus\{x_0\}$
 is $c_1$-uniform with $c_1=c_1(c)$ (This means that $c_1$ is a constant depending only on $c$). Moreover, from the proof of \cite[Lemma 6.7]{Vai6}
 we see that $c_1\leq 9c.$\end{lemma}

We note that each ball $\mathbb{B}(x,r)$ is $2$-uniform and $\mathbb{B}(x,r)\setminus\{x\}$ is $10$-uniform by the proof of \cite[Theorem 6.5]{Vai6}.
By Lemma \ref{star} and  \ref{ThmC}, the following holds.

\begin{lemma}\label{lem3-1}For $x_0\in D$, $\mathbb{B}(x_0,\frac{1}{6}d_D(x_0))\setminus P_D$ is $c_2$-uniform with
$2\leq c_2\leq 18$.\end{lemma}

\begin{lemma}lem\label{lem3-2}For $x,y\in D$, if there is a $c_3$-cone arc $\gamma$ joining $x, y$ in $D$, then for each $w\in \gamma$ the following holds:
$$d_D(w)\geq \frac{1}{2c_3}\min\{d_D(x), d_D(y)\}.$$ Moreover, if $\ell(\gamma[x,w])\leq \ell(\gamma[y,w])$, then $$d_D(w)\geq \frac{1}{2c_3}d_D(x).$$ Otherwise, $$d_D(w)\geq \frac{1}{2c_3}d_D(y).$$
\end{lemma}
\begin{proof}
Let $w_0\in \gamma$ bisect the arclength of $\gamma.$ Obviously, we only need to consider the case $w\in \gamma[x, w_0]$ since the discussion for the case
$w\in \gamma[y, w_0]$ is similar.

If $\ell(\gamma[x,w])\leq \frac{1}{2}d_D(x)$, then
$$d_D(w)\geq d_D(x)-\ell(\gamma[x,w])\geq \frac{1}{2}d_D(x).$$

If $\ell(\gamma[x,w])> \frac{1}{2}d_D(x)$,
then  we have $$ d_D(w)\geq
\frac{1}{c_3}\ell(\gamma[x,w])> \frac{1}{2c_3}d_D(x).$$ The proof is complete.\end{proof}

Let us recall the following result from \cite{LW}.

\begin{lemma}  (\cite[Theorem 1.2]{LW}) \label{ThmD} Suppose that $D_1$ and $D_2$ are convex domains in $E$,
where $D_1$ is bounded and $D_2$ is $c$-uniform, and that there
exist $z_0\in D_1\cap D_2$ and $r>0$ such that $\mathbb{B}(z_0,r) \subset
D_1\cap D_2$. If there exist $R_1>0$ and a constant $c_0>0$ such
that $R_1 \leq c_0r$ and $D_1\subset \overline{\mathbb{B}}(z_0,R_1)$, then
$D_1\cup D_2$ is a $c'$-uniform domain with
$c'=\frac{1}{2}(c+1)(6c_0+1)+c$.\end{lemma}

By Lemma \ref{star},  \ref{ThmA} and \ref{ThmD}, we easily have the following lemma.

\begin{lemma}\label{lw1-3}Let $D\subset E$ be a domain. For $y_1,\; w_1\in D$, if $$\mathbb{B}(y_1,\frac{1}{32}d_D(y_1))\cap \mathbb{B}(w_1,\frac{1}{32}d_D(w_1))\not=\emptyset,$$ then $D_0\setminus P_D$ is a $660c_2^2$-uniform domain, where $D_0=\mathbb{B}(y_1,\frac{1}{16}d_D(y_1))\cup \mathbb{B}(w_1,\frac{1}{32}d_D(w_1))$. \end{lemma}

\section{Stability of John domains and an application}

Before the formulation of our main theorem, we prove a key lemma.

\begin{lemma}\label{coro}Let $D\subset E$ be a domain. For $z_1, z_2\in G=D\setminus P_D,$ let $\gamma$ be a rectifiable arc joining $z_1$ and $z_2$ in $D$. Then there exists an arc $\alpha\subset G$ joining $z_1$ and $z_2$ such that $\ell(\alpha)\leq 660c_2^2 \ell(\gamma)$. Moreover, if $\gamma$ is a $c$-cone arc in $D$, then $\alpha$ is a $(2^{18}c c_2^3+660c_2^2)$-cone arc in $G$, where $c>1$ is a constant and $c_2$ is the  constant from Lemma \ref{lem3-1}.
\end{lemma}

\begin{proof} For given $z_1$ and $z_2$ in $G$, let $\gamma$ be a rectifiable arc joining $z_1$ and $z_2$ in $D$ and let $$U=\{u\in \gamma: d_D(u)>64 d_G(u)\}.$$ If $U=\emptyset$, then let $\alpha_0=\gamma$. Obviously, Lemma \ref{coro} holds.

In the following, we assume that $U\not=\emptyset$. We prove this lemma by considering three cases.

\begin{case}\label{Ca1} There exists some point $w_0\in \gamma$ such that $\{z_1,z_2\}\subset \overline{\mathbb{B}}(w_0, \frac{1}{32}d_D(w_0))$.\end{case}

Then by Lemma \ref{lem3-1}, we know that there is a $c_2$-uniform arc $\alpha_1$ joining $z_1$ and $z_2$ in $G$ which is the desired since $$\ell(\alpha_1)\leq c_2|z_1-z_2|\leq c_2\ell(\gamma).$$

Let $z_0 \in \gamma$ be a point such that
 $\ell(\gamma[z_1,z_0])= \ell(\gamma[z_2,z_0])\,.$
\begin{case}\label{Ca2} For all $w\in \gamma$, $\{z_1,z_2\}\nsubseteq \overline{\mathbb{B}}(w, \frac{1}{32}d_D(w))$, but there is a point $w_1\in \gamma[z_1, z_0]$ such that $z_2\in \overline{\mathbb{B}}(w_1, \frac{1}{32}d_D(w_1))$ or a point $w_2\in \gamma[z_2, z_0]$ such that $z_1\in \overline{\mathbb{B}}(w_2, \frac{1}{32}d_D(w_2))$.
\end{case}

Obviously, we only need to consider the former case since the discussion for the latter case is similar. Without loss of generality, we may assume that $w_1$ is the first point in $\gamma[z_1,z_0]$ along the direction from $z_1$ to $z_0$ such that $z_2\in \overline{\mathbb{B}}(w_1, \frac{1}{32}d_D(w_1))$.

\begin{subcase}\label{Sca1} $U\cap \gamma[z_1,w_1]=\emptyset$.
\end{subcase}

 That is, for all $w\in \gamma[z_1,w_1]$,  $d_D(w)\leq 64d_G(w)$. By Lemma \ref{lem3-1}, there exists a $c_2$-uniform arc $\eta_1$ joining $w_1$ and $z_2$ in $G$. Then we come to prove that $\alpha_2=\gamma[z_1,w_1]\cup\eta_1$ is the desired arc.
By the choice of $\eta_1$, we know that $$\ell(\alpha_2)\leq c_2|w_1-z_2|+\ell(\gamma[z_1,w_1])\leq c_2\ell(\gamma).$$

Assume further that $\gamma$ is a $c$-cone arc. Then we let $u_0$ bisect the arclength of $\eta_1$. If $w\in \gamma[z_1,w_1]$, then $$\ell(\alpha_2[z_1,w])=\ell(\gamma[z_1,w])\leq cd_D(w)\leq 64c d_G(w).$$ If $w\in \eta_1[w_1,u_0]$, then Lemma \ref{lem3-2} yields $$\ell(\alpha_2[z_1,w])=\ell(\gamma[z_1,w_1])+\ell(\eta_1[w_1,w])\leq 64cd_G(w_1)+c_2d_G(w)\leq (128c +1)c_2 d_G(w).$$ If $w\in \eta_1[u_0,z_2]$, then $$\ell(\alpha[z_2,w])=\ell(\eta_1[z_2,w])\leq c_2d_G(w).$$
Hence $\alpha_2$ is the desired.

\begin{subcase}\label{Sca2} $U\cap \gamma[z_1,w_1]\not=\emptyset$.
\end{subcase}

If $z_1\in U$, then let $y_1=z_1$. Otherwise, let $y_1$ be the first point in $\gamma[z_1,w_1]$ along the direction from $z_1$ to $w_1$ such that \begin{equation}\label{eq-new-1}d_D(y_1)=64d_G(y_1).\end{equation} We first consider the case: $$\mathbb{B}(y_1,\frac{1}{32}d_D(y_1))\cap \mathbb{B}(w_1,\frac{1}{32}d_D(w_1))\not=\emptyset.$$ By Lemma \ref{lw1-3}, we know that there is a $660 c_2^2$-uniform arc $\eta_2$  joining $y_1$ and $z_2$ in $G$, then let $\alpha_3=\gamma[z_1,y_1]\cup \eta_2$. Here and in the following, we assume that $\gamma[z_1,y_1]=\{z_1\}$ if $z_1=y_1$.

If $y_1=z_1$, then $\alpha_3=\eta_2$, and obviously, it satisfies Lemma \ref{coro}.
If $y_1\not=z_1$, then replacing $c_2$ by $660 c_2^2$, similar arguments as in Subcase \ref{Sca1} show that $\alpha_3$ is the desired.

In the following, we assume $$\mathbb{B}(y_1,\frac{1}{32}d_D(y_1))\cap \mathbb{B}(w_1,\frac{1}{32}d_D(w_1))=\emptyset$$ and we come to construct an arc $\alpha_4$ satisfying the lemma. We first show the following claim.

\begin{claim}\label{cl-1} There exists a sequence of points $\{y_i\}^{p_1}_{i=1}$ in
$\gamma$, where $p_1\geq 3$ is an odd number, satisfying the following conditions. \begin{enumerate}
\item $y_1=z_1$ or $y_1$ is first point in $\gamma[z_1, w_1]$ from $z_1$ to $w_1$ such that $d_D(y_1)= 64d_G(y_1);$

\item  For each even number $j\in \{1,2,\ldots,p_1\}$, $d_G(y_j)\geq \frac{1}{66}d_D(y_j)$ and $d_D(y_{p_1})\leq 128d_G(y_{p_1})$;

\item If $p_1\geq 5$, then for each even number $j\in \{1,2,\ldots,p_1-2\}$, $y_{j+1}$ is the first point in $\gamma[y_j,w_1]$ from $y_j$ to $w_1$ such that $d_D(y_{j+1})=128d_G(y_{j+1})$;

  \item $p_1$ is the smallest integer with $y_{p_1}\in \mathbb{S}(w_1,\frac{1}{32}d_D(w_1))$ or $\mathbb{B}(y_{p_1},\frac{1}{32}d_D(y_{p_1}))\cap \mathbb{B}(w_1,\frac{1}{32}d_D(w_1))\not=\emptyset$(see Figures \ref{fig1} and \ref{fig1'}).

\end{enumerate}\end{claim}

 For a proof, we let $y_2\in\gamma[y_1,w_1]\cap \mathbb{S}(y_1, \frac{1}{32}d_D(y_1)$ be such that
$$\gamma(y_2,w_1]\cap\mathbb{S}(y_1, \frac{1}{32}d_D(y_1))=\emptyset,$$
where $\gamma(y_2,w_1]$ denote the part $\gamma$ from $y_2$ to $w_1$ such that $y_2\notin \gamma[y_2,w_1]$. Then \begin{equation}\label{eq-33'}d_D(y_2)\leq d_D(y_1)+|y_1-y_2|= \frac{33}{32}d_D(y_1).\end{equation}
By Lemma \ref{star} and \eqref{eq-new-1}, we know that there exists one and only one point, say $x_s$, in $\mathbb{B}(y_1, \frac{1}{6}d_D(y_1))\cap P_D$,  and so
$$d_G(y_2)=|y_2-x_s|\geq |y_2-y_1|-|y_1-x_s|\geq
\frac{1}{64}d_D(y_1),$$

\noindent which, together with \eqref{eq-33'}, shows that
\begin{equation}\label{eq3}d_G(y_2)\geq \frac{1}{66}d_D(y_2).\end{equation}

\begin{figure}[!ht]
\includegraphics{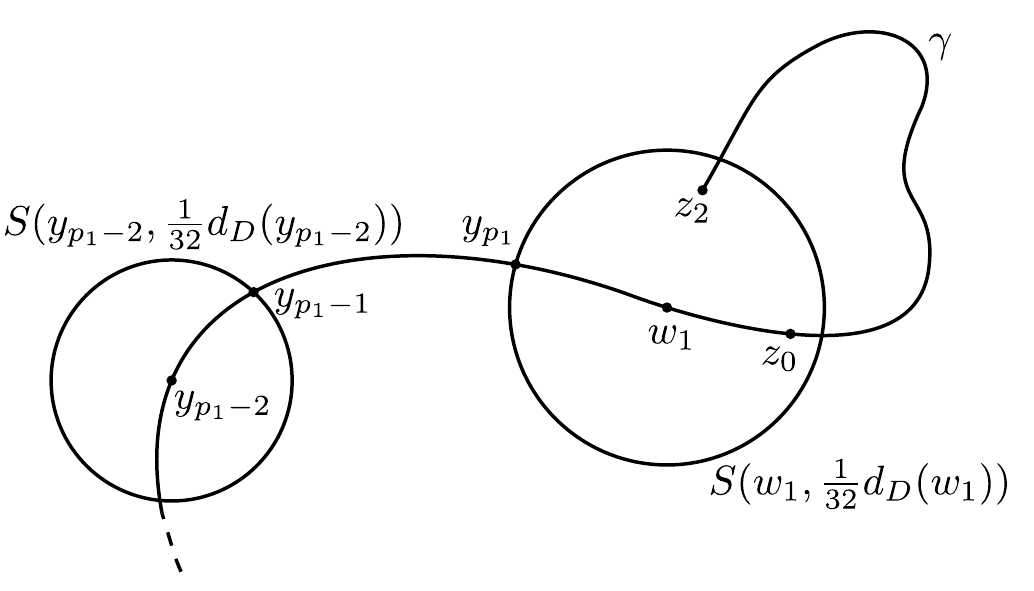} 
\caption{For all $w\in \gamma[y_{p_1-1},w_1]\setminus \mathbb{B}(w_1,\frac{1}{32}d_D(w_1))$, $d_D(w)\leq 128d_G(w)$\label{fig1}}
\end{figure}
\begin{figure}[!ht]
\includegraphics{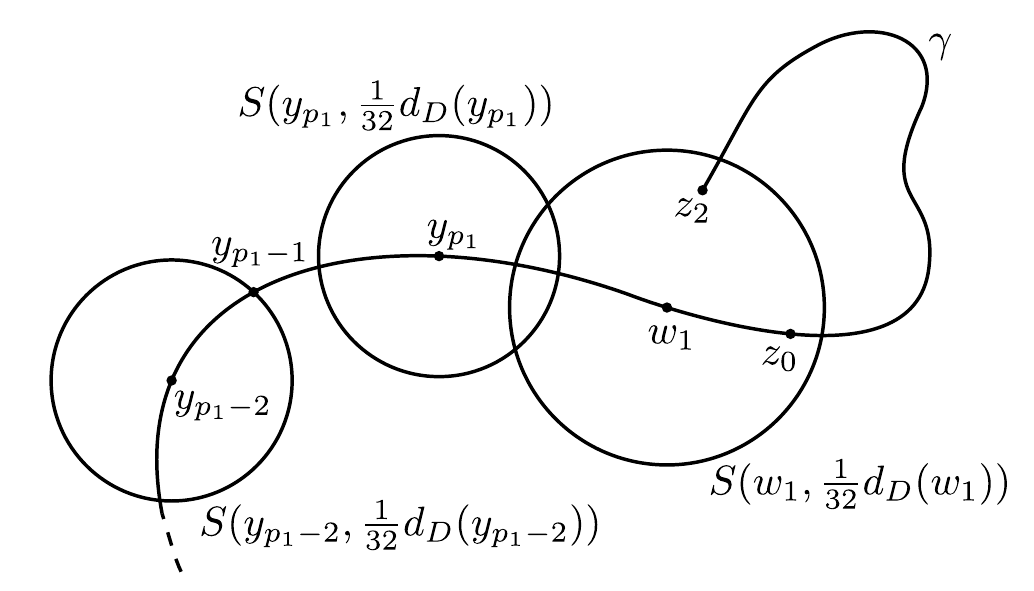}
\caption{$\mathbb{B}(y_{p_1},\frac{1}{32}d_D(y_{p_1}))\cap \mathbb{B}(w_1,\frac{1}{32}d_D(w_1))\not=\emptyset$\label{fig1'}}
\end{figure}
If for all $w\in \gamma[y_2,w_1]\setminus \mathbb{B}(w_1,\frac{1}{32}d_D(w_1))$, $d_D(w)\leq 128d_G(w)$, then the claim obviously holds by letting $y_3\in \gamma[y_2,w_1]\cap \mathbb{S}(w_1,\frac{1}{32}d_D(w_1)$ and $p_1=3$.

If there is some  $w_0\in \gamma[y_2,w_1]\setminus \mathbb{B}(w_1,\frac{1}{32}d_D(w_1)$ such that $d_D(w_0)> 128d_G(w_0)$, then let $y_3$ be the first point in $\gamma[y_2,w_1]$ from $y_2$ to $w_1$ such that $d_D(y_3)=128d_G(y_3)$. If $\mathbb{B}(y_3,\frac{1}{32}d_D(y_3))\cap \mathbb{B}(w_1,\frac{1}{32}d_D(w_1))\not=\emptyset$, then the claim holds and $p_1=3$. Otherwise, let $y_4\in\gamma[y_3,w_1]\cap \mathbb{S}(y_3, \frac{1}{32}d_D(y_3)$ be such that
$$\gamma(y_4,w_1]\cap\mathbb{S}(y_3, \frac{1}{32}d_D(y_3))=\emptyset.$$

Then by Lemma \ref{star}, and a similar argument as in the proof of  \eqref{eq3}, we have
$$d_G(y_4)\geq \frac{1}{66}d_D(y_4).$$

If for all $w\in \gamma[y_4,w_1]\setminus \mathbb{B}(w_1,\frac{1}{32}d_D(w_1))$, $d_D(w)\leq 128d_G(w)$, then we complete the proof of the claim by letting $y_5\in \gamma[y_4,w_1]\cap \mathbb{S}(w_1,\frac{1}{32}d_D(w_1)$.
$\cdots$.
%

 By repeating this process for finite steps, we get a sequence $\{y_i\}^{p_1}_{i=1}\in
\gamma$ satisfying Claim \ref{cl-1}, where $p_1<\frac{M}{\log\frac{33}{32}}$, since for each $i\in\{1,2,\ldots,\frac{p_1-1}{2}\}$, $$\ell_{k_D}(\gamma[y_{2i-1}, y_{2i}])\geq \log\big(1+\frac{|y_{2i-1}-y_{2i}|}{d_D(y_{2i-1})}\big)=\log\frac{33}{32},$$ and $M=\ell_{k_D}(\gamma[z_1,z_2])$.
  Hence Claim \ref{cl-1} holds.\medskip

 We continue the construction of $\alpha_4$. Let $\gamma_1=\gamma[z_1,y_1]$ and for each $j\in \{2, \ldots, \frac{p_1+1}{2}\}$, let $\gamma_j=\gamma[y_{2j-2},y_{2j-1}]$. By Lemma \ref{lem3-1}, we know that for each $j\in \{1,2,\ldots,\frac{p_1-1}{2}\}$,
there exists a $c_2$-uniform arc $\beta_{j}\subset G$ joining $y_{2j-1}$ and
$y_{2j}$. By Lemmas \ref{lem3-1} and \ref{lw1-3}, there exists a $660c_2^2$-uniform arc $\eta_3$ joining $y_{p_1}$ and $z_2$. Take
 $$\alpha_4=\gamma_1\cup\beta_1\cup\gamma_2\cup\ldots\cup\beta_{\frac{p_1-1}{2}}\cup\gamma_{\frac{p_1+1}{2}}\cup\eta_3.$$ Now, we come to show that $\alpha_4$ is the desired arc.

First observe that \begin{equation}\label{neweq-1}\ell(\alpha_4)\leq 660c^2_2\ell(\gamma).\end{equation}
To prove that $\alpha_4$ is a cone arc in $G$, it is enough to show that $$\min\{\ell(\alpha_4[z_1,w]),\ell(\alpha_4[w,z_2])\}\leq (2^{18}c c_2^3+660 c_2^2) d_G(w).$$

 If $w\in \gamma_1\cup \gamma_2\cup\ldots\cup\gamma_{\frac{p_1+1}{2}}$,  then from the assumption, Claim \ref{cl-1} and \eqref{neweq-1}, the above inequality obviously holds.

For the case where $w\in \beta_1\cup\ldots \cup \beta_{\frac{p_1-1}{2}}$, we see that there exists some $i\in \{1,2,\ldots,\frac{p_1-1}{2}\}$ such that  $w\in \beta_i$.
By Claim \ref{cl-1},
$$d_G(y_{2i})\geq \frac{1}{66}d_D(y_{2i})\geq\frac{1}{66}(d_{D}(y_{2i-1})-|y_{2i-1}-y_{2i}|)\geq \frac{62}{33}d_G(y_{2i-1}),$$ whence Lemma \ref{lem3-2} yields \begin{equation}\label{neweq-6}d_G(w)\geq \frac{1}{2c_2}\min\{d_G(y_{2i-1}),d_G(y_{2i})\} =\frac{1}{2c_2}d_G(y_{2i-1}),\end{equation} which, together with Claim \ref{cl-1}, leads to
\begin{eqnarray*}\ell(\alpha_4[z_1,w])&\leq&  c_2\ell(\gamma[z_1,y_{2i-1}])+c_2|y_{2i-1}-y_{2i}|\\ &\leq& 4c_2(32c +1)d_G(y_{2i-1})\leq 8 c^2_2(32c+1)d_G(w).\end{eqnarray*}

For the remaining case where $w\in \eta_3$, we let $u_0$ bisect the arclength of $\eta_3$. If $w\in \eta_3[z_2,u_0]$, then, obviously, $$\ell(\alpha_4[z_2,w])\leq c_2d_G(w).$$
If $w\in \eta_3[y_3,u_0]$, then by Lemma \ref{lem3-2} and Claim \ref{cl-1}, we have \begin{eqnarray*}\ell(\alpha_4[z_1,w])&\leq & c_2\ell(\gamma[z_1,y_{p_1}])+\ell(\eta_3[y_{p_1},w])\\ &\leq& 128c c_2d_G(y_{p_1})+660 c^2_2d_G(w)\leq (2^{18}c c_2^3+660 c^2_2)d_G(w).\end{eqnarray*}

\begin{case}\label{Ca3} $z_1\notin \bigcup_{w\in \gamma[z_2,z_0]}\overline{\mathbb{B}}(w, \frac{1}{32}d_D(w))$ and $z_2\notin \bigcup_{w\in \gamma[z_1,z_0]}\overline{\mathbb{B}}(w, \frac{1}{32}d_D(w))$.
\end{case}

We may assume that $U\cap \gamma[z_1,z_0]\not=\emptyset$. Then by similar discussions as in the proof of Claim \ref{cl-1}, we get the following Claim.

\begin{claim}\label{cl-2} There exists a sequence of points $\{u_i\}^{p_2}_{i=1}$ in $\gamma$, where $p_2\geq 2$ is an integer, satisfying the following conditions.  \begin{enumerate}
\item $u_1=z_1$ or $u_1$ is first point in $\gamma[z_1, z_0]$ from $z_1$ to $z_0$ such that $d_D(u_1)= 64d_G(u_1);$

\item  For each even number $j\in \{1,2,\ldots,p_2\}$, $d_D(u_j)\leq 66d_G(u_j)$ and if $p_2$ is an odd number, then $d_D(u_{p_2})\leq 128d_G(u_{p_2})$;

\item If $p_2\geq 4$, then for each even number $j\in \{1,2,\ldots,p_2-2\}$, $u_{j+1}$ is the first point in $\gamma[u_j,z_2]$ from $u_j$ to $z_2$ such that $d_D(u_{j+1})=128d_G(u_{j+1})$;

  \item $p_2$ is the smallest integer such that $u_{p_2}=z_0$ or $z_0\in \overline{\mathbb{B}}(u_{p_2-1},\frac{1}{32}d_D(u_{p_2-1}))$.

\end{enumerate}\end{claim}

\begin{figure}[!ht]
\includegraphics[width=0.55\textwidth]{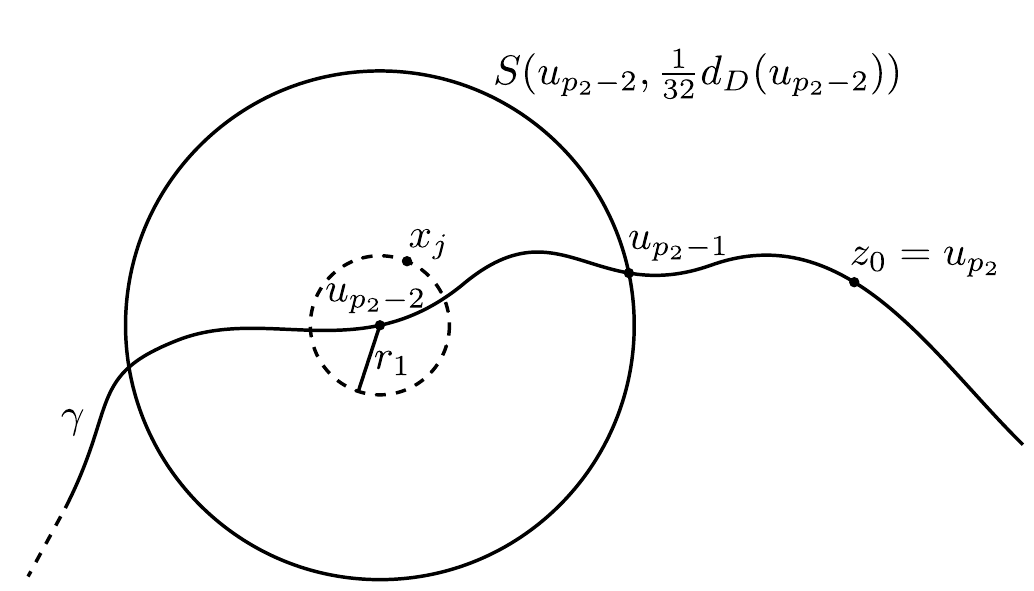} 
\caption{$r_1=\frac{1}{128}d_D(u_{p_2-2})$ and $x_j\in P_D$\label{fig2}}
\end{figure}
\begin{figure}[!ht]
\includegraphics[width=0.65\textwidth]{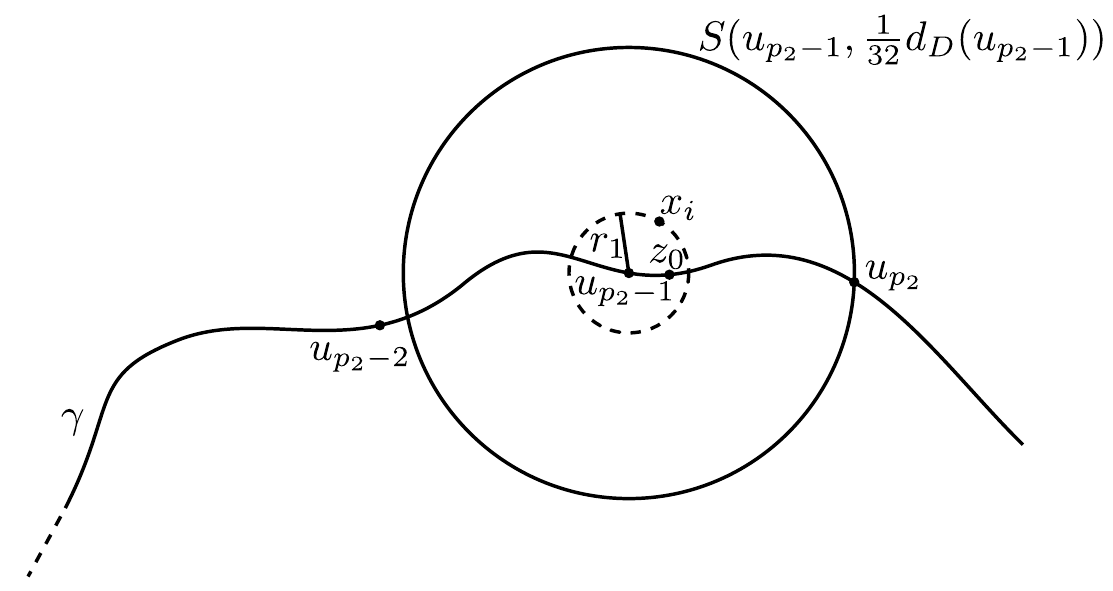}
\caption{$r_1=\frac{1}{128}d_D(u_{p_2-1})$ and $x_i\in P_D$\label{fig3}}
\end{figure}

We note that there are two possibilities for  $u_{p_2}$ (see Figures \ref{fig2} and \ref{fig3}) : One is $u_{p_2}=z_0$ and for all $w\in \gamma[u_{p_2-1},u_{p_2}]$, $d_D(w)\leq 128 d_G(w)$; and the other is $u_{p_2}\in \overline{\mathbb{B}}(u_{p_2-1},\frac{1}{32}d_D(u_{p_2-1}))\cap \gamma[z_0,z_2]$. No matter in which case, the proof is similar. So, in the following, we assume that $u_{p_2}\in \overline{\mathbb{B}}(u_{p_2-1},\frac{1}{32}d_D(u_{p_2-1}))\cap \gamma[z_0,z_2]$. Then $p_2$ is an even number, and by Claim \ref{cl-2}, we note that \begin{equation}\label{neweq-5}d_D(u_{p_2})\leq 66 d_G(u_{p_2}).\end{equation} We assume that $U\cap \gamma[z_2,u_{p_2}]\not=\emptyset$. Then by similar discussions as in the proof of Claim \ref{cl-1}, we also get the following claim.

\begin{claim}\label{cl-3} There exists a sequence of points $\{v_i\}^{p_3}_{i=1}$ in
 $\gamma[z_2,u_{p_2}]$, where $p_3\geq 2$ is an integer, satisfying the following conditions.  \begin{enumerate}
\item $v_1=z_2$ or $v_1$ is the first point in $\gamma[z_2, u_{p_2}]$ from $z_2$ to $u_{p_2}$ such that $d_D(v_1)= 64d_G(v_1);$

\item  For each even number $j\in \{1,2,\ldots,p_3\}$, $d_D(v_j)\leq 66d_G(v_j)$ and  $d_D(u_{p_3})\leq 66d_G(u_{p_3})$;

\item If $p_3\geq 4$, then for each even number $j\in \{1,2,\ldots,p_3-2\}$, $v_{j+1}$ is the first point in $\gamma[v_j,u_{p_2}]$ from $v_j$ to $u_{p_2}$ such that $d_D(v_{j+1})=128d_G(v_{j+1})$;

  \item $p_3$ is the smallest integer such that $u_{p_2}=v_{p_3}$ or $u_{p_2}\in \overline{\mathbb{B}}(v_{p_3-1},\frac{1}{32}d_D(v_{p_3-1}))$.

\end{enumerate}\end{claim}

We note that what we consider here is the case when $p_2$ is an even number, then by using the similar method as in the discussion of Case \ref{Ca2}, together with Lemma \ref{lem3-1}, Claims \ref{cl-2} and  \ref{cl-3}, we construct an arc $\alpha_5=\alpha_5'\cup\alpha^{''}_5$ in $G$ such that $$\alpha_5'=\gamma[z_1,u_1]\cup\beta_{1,1}[u_1,u_2]\cup\ldots\cup\beta_{1,i}[u_{2i-1}, u_{2i}]\cup\gamma[u_{2i},u_{2i+1}]\cup\ldots\cup\beta_{1,\frac{p_2}{2}}[u_{p_2-1},u_{p_2}]$$ and if $p_3$ is an odd number, then $$\alpha_5^{''}=\gamma[z_2,v_1]\cup\beta_{2,1}[v_1,v_2]\cup\ldots\cup\beta_{2,i}[v_{2i-1},v_{2i}]\cup\gamma[v_{2i},v_{2i+1}]\cup\ldots\cup\gamma[v_{p_3-1},u_{p_2}];$$ if $p_3$ is en even number, then $$\alpha_5^{''}=\gamma[z_2,v_1]\cup\beta_{2,1}[v_1,v_2]\cup\ldots\cup\beta_{2,i}[v_{2i-1},v_{2i}]\cup\gamma[v_{2i},v_{2i+1}]\cup\ldots\cup\beta_{2,\frac{p_3}{2}}[v_{p_3-1},u_{p_2}],$$ where $\beta_{1,i}[u_{2i-1}, u_{2i}]$ and $\beta_{2,i}[v_{2i-1},v_{2i}]$ are $c_2$-uniform arcs in $G$.

Obviously, $$\ell(\alpha_5)\leq c_2 \ell(\gamma).$$ Moreover, if $\gamma$ is a $c$-cone arc, we prove that for all $w\in \alpha_5$, $$\min\{\ell(\alpha_5[z_1,w]),\ell(\alpha_5[w,z_2])\}\leq 8 c^2_2(32c+1) d_G(w).$$ To show this, we only need to consider the case $w\in \alpha_5'$.

If $w\in \gamma\cap\alpha_5'$, then $$\ell(\alpha_5[z_1,w])\leq c_2\ell(\gamma[z_1,w])\leq 128c c_2d_G(w).$$

If $w\in \beta_{1,1}[u_1,u_2]\cup \ldots\cup\beta_{1,\frac{p_2}{2}}[u_{p_2-1},u_{p_2}]$, then there exists some $i\in\{1,2,\ldots,\frac{p_2}{2}\}$ such that $w\in \beta_{1,i}$. Hence by Lemma \ref{lem3-2}, Claim \ref{cl-2}, \eqref{neweq-6} and \eqref{neweq-5}, we have for all $w\in \beta_{1,i}[u_{2i-1},u_{2i}]$, \begin{eqnarray*}\ell(\alpha_5[z_1,w])&\leq& c_2\ell(\gamma[z_1,u_{2i-1}])+\ell(\beta_{1,i}[u_{2i-1},w])\\ &\leq& 128c c_2d_G(u_{2i-1})+c_2|u_{2i-1}-u_{2i}|\\&\leq& 4c_2(32c+1)d_G(u_{2i-1})\leq 8 c^2_2(32c+1)d_G(w).\end{eqnarray*}
Hence, Lemma \ref{coro} is follows from Cases \ref{Ca1}, \ref{Ca2} and \ref{Ca3}.
\end{proof}

\begin{theorem}\label{thm1} A domain $D\subset E$ is a $c$-John domain if and only if
$G=D\setminus P_D$ is a $c_1$-John domain, where $c\geq 1$ and $c_1\geq 1$
depend only on each other.\end{theorem}

\medskip
\begin{proof}~{\em Necessity:~~}Let $D$ be a $c$-John domain. For given $z_1,z_2\in G$, there is a $c$-cone arc $\gamma$ in $D$
 joining $z_1$ and $z_2$. By Lemma \ref{coro}, we know that there exists a $(2^{18}cc_2^3+660c c_2^2)$-cone arc in $G$ joining $z_1$ and $z_2$. Then $G$ is a John domain.

{\em Sufficiency:}\;  Let $c= \frac{65}{63}c_1.$
For each $z_1,z_2\in D$, we prove that there exists an arc $\beta\subset D$ joining $z_1$ and $z_2$ such that
\begin{equation}\label{thm1-mianeq2}\min\{\ell(\beta[z_1,w]),\ell(\beta[z_2,w])\}\leq c d_D(w)\;\;{\rm for}\;{\rm all}\; w\in \beta.\end{equation}
If $|z_1-z_2|\leq
\frac{1}{4}\max\{d_D(z_1),d_D(z_2)\}$, then let $$\beta=[z_1,z_2],$$ and obviously, $$\min\{|z_1-w|,|z_2-w|\}\leq d_D(w)\;\;{\rm for}\;{\rm all}\; w\in [z_1,z_2],$$ which shows that \eqref{thm1-mianeq2} holds.

 In the following, we assume that $|z_1-z_2|>
\frac{1}{4}\max\{d_D(z_1),d_D(z_2)\}$. If $z_1, z_2\in G$, then
let $\gamma$ be a $c_1$-cone arc joining $z_1$ and $z_2$ in
$G$, and take $$\beta=\gamma.$$ Then $\beta$ satisfies \eqref{thm1-mianeq2} since $G\subset D$.

If $z_1\in P_D$ but $z_2\notin P_D$, then let $x\in G$ be such that $|z_1-x|=
\frac{1}{64}d_D(z_1)$, and $\gamma$ be a $c_1$-cone arc
joining $x$ and $z_2$ in $G$. Take
$$\beta=[z_1,x]\cup\gamma.$$
If $z_1\notin P_D$ but $z_2\in P_D$, then let $y\in G$ be such that $|z_2-y|=
\frac{1}{64}d_D(z_2)$, and $\gamma$ be a $c_1$-cone arc
joining $y$ and $z_1$ in $G$. Take
$$\beta=[z_2,y]\cup\gamma.$$
If $z_1,\;z_2\in P_D$, then let $x\in G$ such that $|z_1-x|=
\frac{1}{64}d_D(z_1)$ and $y\in G$ such
that $|z_2-y|= \frac{1}{64}d_D(z_2)$, and $\gamma$ be a
$c_1$-cone arc joining $x$ and $y$ in $G$. Take
$$\beta=[z_1,x]\cup \gamma\cup[y,z_2].$$ To prove that these three arcs $\beta$ are cone arcs in $D$, it is enough to consider the third case where $z_1\;z_2\in P_D$. In this case, $\beta=[z_1,x]\cup \gamma\cup[y,z_2].$

Let $z_0$ bisect the arclength of $\gamma$. It suffices to prove that for all $w\in
\beta[z_1,z_0]$, $$\ell(\beta[z_1,w])\leq \frac{65}{63}c_1
d_D(w).$$

On one hand, if $w\in[z_1,x]$, then \begin{equation}\label{eq10} \ell(\beta[z_1,w])=|z_1-w|\leq
\frac{1}{64}d_D(z_1)\leq \frac{1}{63}d_D(w).\end{equation}

On the other hand, if $w\in \gamma[x,z_0]$, then Lemma \ref{lem3-2}
shows $$d_D(w)\geq \frac{1}{2c_1}d_D(x),$$ which, together with
\eqref{eq10}, shows that

$$\ell(\beta[z_1,w])=|z_1-x|+\ell(\gamma[x,w])\leq \frac{1}{63}d_D(x)+ c_1d_D(w)\leq \frac{65}{63}c_1d_D(w).$$
Hence \eqref{thm1-mianeq2} holds, and so the proof of Theorem \ref{thm1} is complete.
\end{proof}

As an application of Theorem \ref{thm1}, we show the following result concerning inner uniform
domains.

\begin{theorem}\label{thm2}A domain $D\subset E$ is an inner $c$-uniform domain if and only if
$G=D\setminus P_D$ is  an inner $c_1$-uniform domain, where $c\geq 1$ and
$c_1\geq 1$ depend only on each other.\end{theorem}

\begin{proof}
We first prove the necessary part of the theorem, that is, if $D$ is
an inner $c$-uniform domain,  we need to prove that  each pair
$z_1,z_2\in G$  can be joined  by an inner
$c_1$-uniform arc in $G$, where $c_1=2^{18}c^2c_2^3+660c_2^2$, and $c_2$ ($2\leq c_2\leq 18$) is a constant from Lemma \ref{lem3-1} .

For $z_1,z_2\in G$, since $D$ is an inner $c$-uniform domain, then there is an arc $\gamma$ joining $z_1$ and $z_2$ in $D$ such that for all $w\in \gamma$
$$\min\{\ell(\gamma[z_1,w]),\ell(\gamma[z_2,w]) \} \leq c d_D(w)$$
and
$$\ell(\gamma)\leq c \lambda_D(z_1,z_2).$$
By Lemma \ref{coro}, we know that there exists an arc
$\alpha\subset G$ such that $\alpha$ is a $(2^{18}c^2c_2^3+660c_2^2)$-cone arc in $G$ and $\ell(\alpha)\leq 660c^2_2 \ell(\gamma)$. Hence
$$\ell(\alpha)\leq 660c^2_2 \ell(\gamma)\leq 660 c c^2_2\lambda_D(z_1,z_2)\leq 660c c^2_2 \lambda_G(z_1,z_2),$$ which shows that $\alpha$ is the desired arc.

\medskip

To prove the sufficient part of Theorem \ref{thm2}, we need to prove that for each $z_1,z_2\in D$, there exists an arc $\beta$ joining $z_1$ and $z_2$ in $D$ such that
\begin{equation}\label{thm2-mianeq1}\min\{\ell(\beta[z_1,w]),\ell(\beta[z_2,w])\}\leq (1485c_1 c_2^2+\frac{1}{8})d_D(w)\;\;{\rm for}\;{\rm all}\;w\in\beta,\end{equation}
and \begin{equation}\label{thm2-mianeq2}\ell(\beta)\leq (1485c_1 c_2^2+\frac{1}{8})\lambda_D(z_1,z_2).\end{equation}

If $|z_1-z_2|\leq\frac{1}{4}\max\{d_D(z_1),d_D(z_2)\}$, then let $$\beta=[z_1,z_2].$$ Obviously, $\beta$ satisfies  \eqref{thm2-mianeq1} and \eqref{thm2-mianeq2}.

 In the following, we assume that \begin{equation}\label{eqa}|z_1-z_2|>
\frac{1}{4}\max\{d_D(z_1),d_D(z_2)\}.\end{equation} We divide the proof of this case
into two parts.

\begin{case}\label{ca3}$z_1,z_2\in G$.\end{case}
Since $G$ is an inner $c_1$-uniform domain,
then there is a $c_1$-cone arc $\gamma$ joining $z_1$ and $z_2$ in
$G$ such that \begin{equation}\label{eq001}\ell(\gamma)\leq
c_1\lambda_G(z_1,z_2).\end{equation}

Obviously, $\gamma$ satisfies \eqref{thm2-mianeq1} since $G\subset D$. In order to prove $\gamma$ satisfies \eqref{thm2-mianeq2}, we let $\alpha$ be an arc joining $z_1$ and $z_2$ in $D$ with
\begin{equation}\label{eq1-11}\ell(\alpha)\leq 2\lambda_D(z_1,z_2).\end{equation}
By Lemma \ref{coro}, we join $z_1$ and $z_2$ by an arc
$\alpha_1\subset G$ such that
$$\ell(\alpha_1)\leq 660c^2_2 \ell(\alpha),$$ which,
together with \eqref{eq001} and \eqref{eq1-11}, shows
that $$\ell(\gamma)\leq c_1\lambda_G(z_1,z_2)\leq c_1 \ell(\alpha_1)\leq 1320c_1 c^2_2\lambda_D(z_1,z_2).$$

Now we take $$\beta=\gamma.$$ Obviously, $\beta$ satisfies \eqref{thm2-mianeq1} and \eqref{thm2-mianeq2}.

\begin{case}\label{ca4}$z_1\notin G$ or $z_2\notin G$.\end{case}

Without loss of generality, we may assume that $z_1\notin G$ and $z_2\notin G$, since the proof for the case $z_1\in G$, $z_2\notin G$ or $z_1\notin G$, $z_2\in G$ is similar.
Let $x,\;y\in G$ be such that \begin{equation}\label{eqa1}|z_1-x|=
\frac{1}{64}d_D(z_1),\;\;\;\;\;\;|z_2-y|= \frac{1}{64}d_D(z_2),\end{equation}  and let
$\gamma$ be an inner $c_1$-uniform arc joining $x$ and $y$ in $G$.
Take $$\beta=[z_1,x]\cup \gamma\cup[y,z_2].$$

By Theorem \ref{thm1} and its proof, we know that $\beta$ satisfies \eqref{thm2-mianeq1}.
It follows from Case \ref{ca3} that $$\ell(\gamma)\leq 1320c_1c_2^2
c_2\lambda_D(x,y),$$ which, together with \eqref{eqa} and \eqref{eqa1}, shows that
\begin{eqnarray*}\ell(\beta[z_1,z_2])&=&|z_1-x|+
\ell(\gamma[x,y])+|y-z_2|\\&\leq& \frac{1}{8}|z_1-z_2|+1320c_1
c^2_2\lambda_D(x,y)\\&\leq&(1485c_1
c^2_2+\frac{1}{8})\lambda_D(z_1,z_2),
\end{eqnarray*}from which we see that $\beta$ satisfies \eqref{thm2-mianeq2}.
Hence the proof of Theorem \ref{thm2} is complete.
\end{proof}

\section{Stability of $\psi$-John domains}

In \cite{HV}, the authors considered the $\psi$-John domains which is a generalization of John domains.

\begin{definition}\label{def-2'}A domain $D$ is said to be a $\psi$-John domain if $\psi$ is an increasing self-homeomorphism of $[0,\infty]$ and if for some fixed $x_0\in D$ and for all $y\in D$,  we have $$k_D(x_0,y)\leq \psi\left(\frac{|x_0-y|}{\min\{d_D(x_0),d_D(y)\}}\right).$$

\end{definition}

The following lemma follows immediately from  \eqref{eq(0000)}.

\begin{lemma}\label{lem4-0}If $\psi:[0,\infty]\to [0,\infty]$ is a homeomorphism  such that a domain is a $\psi$-John domain, then $\log(1+t)\leq \psi(t)$ holds for all $t\geq 0.$ \end{lemma}

By \cite[Theorem 2.23]{Vai4}, we have the following lemma which is useful for the discussion in the rest of this section.

\begin{lemma}\label{lem4-1}Suppose that $D\subset E$ is a domain and that $D_1\subset D$ is a $c$-uniform domain. Then for all $x, y\in D_1$,
$$k_D(x,y)\leq c_1 j_D(x,y)$$ with $c_1=c_1(c)\leq 7c^3.$\end{lemma}

From Lemma \ref{lem3-1} and Lemma \ref{lem4-1}, we easily get the
following corollary.

\begin{corollary}\label{cor4-2}Suppose that $D\subset E$ is
a domain and $ G=D\setminus P_D$. For $x,y\in D$, if
$d_D(x)=128d_G(x)$ and $y\in \overline{\mathbb{B}}(x,
\frac{1}{32}d_D(x))$,  then $k_G(x,y)\leq \mu j_G(x,y),$ where
$\mu\leq 7\times 18^3$ is a constant. \end{corollary}

Meanwhile, \cite[Lemma 3.7(2)]{Vu} yields the following corollary.

\begin{corollary}\label{cor4-1}Suppose that $D\subset E$ is a domain. For
$x,y\in D$, if $|x-y|\leq \frac{1}{2}\min\{d_D(x),d_D(y)\}$, then
$k_D(x,y)\leq 2 j_D(x,y)$. \end{corollary}

Before the statement of our main result in this section, we prove the following two lemmas.

\begin{lemma}\label{lem4-2}Let $D$ be a domain and $ G=D\setminus P_D$. For each $x\in D$, there exists some point $w\in \mathbb{S}(x, \frac{1}{32}d_D(x))$ such that $$\frac{1}{48}d_D(x)<\frac{1}{33}d_D(w)\leq d_{G}(w)\leq \frac{33}{31}d_D(w).$$ \end{lemma}
\begin{proof} Let $x\in D$. By Lemma \ref{star}, there exists at most one point in $P_D\cap \mathbb{B}(x,\frac{1}{6}d_D(x))$. On one hand, if $P_D\cap \mathbb{B}(x,\frac{1}{6}d_D(x))=\emptyset,$ let $w\in \mathbb{S}(x, \frac{1}{32}d_D(x))$. On the other hand, if $P_D\cap \mathbb{B}(x,\frac{1}{6}d_D(x))\not=\emptyset,$ then there exists one and only one point $x_i$ in $P_D\cap \mathbb{B}(x,\frac{1}{6}d_D(x))$.  Let $l$ be a line determined by $x$ and $x_i$, and take $w\in l\cap \mathbb{S}(x,\frac{1}{32}d_D(x))$ such that $d_G(w)\geq \frac{1}{32}d_D(x)$. Then $$d_D(w)\leq d_D(x)+|w-x|\leq \frac{33}{32}d_D(x),$$and so $$d_G(w)\geq\frac{1}{32}d_D(x)\geq \frac{1}{33}d_D(w).$$ Hence $$d_D(w)\geq d_D(x)-|x-w|=\frac{31}{32}d_D(x)\geq \frac{31}{33}d_G(w)$$ and $$d_G(w)\geq \frac{1}{33}d_D(w)>\frac{1}{48}d_D(x).$$ The proof is complete. \end{proof}

\begin{lemma}\label{lem4-3}Let $D$ be a domain and $ G=D\setminus P_D$. For each $x\in D$ and $w\in \mathbb{S}(x, \frac{1}{32}d_D(x))$, if $d_D(x)\geq 128d_G(x)$, then  $ d_{G}(w)\geq \frac{1}{44}d_D(w).$ \end{lemma}
\begin{proof} Observe first that $$d_D(w)\leq d_D(x)+|w-x|\leq \frac{33}{32}d_D(x).$$ Let $x\in D$. Since $d_D(x)\geq 128d_G(x)$, then by Lemma \ref{star}, there exists one and only one point, namely $x_i$, in $P_D\cap \mathbb{B}(x,\frac{1}{6}d_D(x))$.  Hence $$d_G(w)=|w-x_i|\geq |x-w|-|x-x_i|\geq \frac{3}{128}d_D(x)\geq \frac{1}{44}d_D(w).$$   Thus the proof of the lemma is complete. \end{proof}

For $\psi$-John domains, we get the following stability of $\psi$-John domain.

\begin{theorem}\label{thm3}A domain $D\subset E$ is a $\psi$-John domain if and only if
$G=D\setminus P_D$ is  a $\psi_1$-John domain, where $\psi$ and
$\psi_1$ are homeomorphisms depending only on each other.\end{theorem}

\begin{proof}
We first prove the necessary part of the theorem. For this, we assume that $D$ is $\psi$-John domain with center $x_0$, where $x_0\in D$. By Lemma \ref{lem4-2}, there exists some point $w_0$ in $ \mathbb{S}(x_0,\frac{1}{32}d_D(x_0))$ such that \begin{equation}\label{th-eq1}\frac{1}{48}d_D(x_0)<\frac{1}{33}d_D(w_0)\leq d_{G}(w_0)\leq \frac{33}{31}d_D(w_0)\end{equation}  and \begin{equation}\label{th-eq1'}d_D(w_0)\leq d_D(x_0)+|x_0-w_0|\leq \frac{33}{32}d_D(x_0).\end{equation} We come to prove that there exists some homeomorphism $\psi'$ of $[0,\infty)$ such that $G$ is a $\psi'$-John domain with center $w_0$.  That is, we need to find a homeomorphism $\psi'$ of $[0,\infty)$ such that for each $y\in G$,\begin{equation}\label{eqjohn}k_G(w_0,y)\leq \psi'\left(\frac{|w_0-y|}{\min\{d_G(w_0),d_G(y)\}}\right).\end{equation}


%

For $y\in G$, if $|w_0-y|\leq \frac{1}{2}\max\{d_G(w_0),d_G(y)\}$, then Lemmas \ref{lem4-0} and Corollary \ref{cor4-1} show that

$$k_G(w_0,y)\leq 2 \log\left(1+\frac{|w_0-y|}{\min\{d_G(w_0),d_G(y)\}}\right)
\leq 2\psi\left(\frac{|w_0-y|}{\min\{d_G(w_0),d_G(y)\}}\right),$$
which shows that \eqref{eqjohn} holds with $\psi_1(t)=2\psi(t).$
  Hence, in
the following, we assume that \begin{equation}\label{theq00}|w_0-y|>
\frac{1}{2}\max\{d_G(w_0),d_G(y)\}.\end{equation}

Let $\gamma$ be a $2$-neargeodesic joining $w_0$ and $y$ in $D$. We leave the proof for a moment and prove the following claim.

\begin{claim}\label{cl-4}  There exists a sequence of points $\{w_i\}^{p}_{i=0}$ in $\gamma$, where $p\geq 1$ is an integer, satisfying the following conditions.  \begin{enumerate}

\item  For each even number $j\in \{0,\ldots,p-1\}$, $d_D(w_j)\leq 44d_G(w_j)$;

\item For each even number $j\in \{0,\ldots,p-1\}$, $w_{j+1}$ is the first point in $\gamma[w_j,y]$ from $w_j$ to $y$ such that $d_D(w_{j+1})=128d_G(w_{j+1})$;

  \item If $p\geq 2$, then for each even number $j\in \{1,\ldots,p\}$, $w_j\in \overline{\mathbb{B}}(w_{j-1},\frac{1}{32}d_D(w_{j-1}))$.

\end{enumerate}\end{claim}

Obviously, by \eqref{th-eq1}, we have $$d_D(w_0)\leq 33d_G(w_0)<44d_G(w_0).$$ If for all $w\in \gamma$, $d_D(w)< 128 d_G(w)$, then let $w_1=y$. Then the claim obviously holds with $p=1.$
If there exists some point $v_0\in \gamma$
such that $d_D(v_0)\geq 128d_G(v_0),$ then by \eqref{th-eq1}, there exist a point $w_1\in \gamma$  be the first point from $w_0$ to
$y$ satisfying $$d_D(w_1)= 128d_G(w_1).$$
 If $y\in \overline{\mathbb{B}}(w_1,\frac{1}{32}d_D(w_1))$, then the claim holds by letting $w_2=y$, and then $p=2$. Otherwise,
 let $w_2\in \gamma\cap \mathbb{S}(w_1,\frac{1}{32}d_D(w_1))$   such that
$$\gamma[w_2,y]\cap \mathbb{B}(w_1, \frac{1}{32}d_D(w_1))=\emptyset.$$
Then by Lemma \ref{lem4-3}, we have $$d_G(w_2)\geq   \frac{1}{44}d_D(w_2).$$
 If for all $w\in \gamma[w_2,y]$, $$d_D(w)\leq 128d_G(w),$$
then the claim holds with $w_3=y$, and then $p=3$. Otherwise, let $w_3$ be the first point in $\gamma[w_2,y]$ from $w_2$ to $y$ such that
 $$d_G(w_3)=\frac{1}{128}d_D(w_3).$$
 By repeating this process for finite steps, we get a sequence  $\{w_i\}^p_{i=0}\in
\gamma$  satisfying the claim,
 where $p<\frac{M}{\log\frac{33}{32}}$, since for each even number $i\in\{1,2,\ldots,p\}$, $$\ell_{k_D}(\gamma[w_{i-1}, w_{i}])\geq \log\big(1+\frac{|w_{i-1}-w_{i}|}{d_D(w_{i-1})}\big)=\log\frac{33}{32},$$ and $M=\ell_{k_D}(\gamma[w_0,y])$. Hence Claim \ref{cl-4} holds.

Now, we come back to the proof of the necessary part of the theorem.
By Claim \ref{cl-4}, we know that for each even number $j\in
\{0,\ldots,p-1\}$ the following holds: for all $w\in
\gamma[w_j,w_{j+1}]$, $$d_D(w)\leq 128 d_G(w).$$ Hence
\begin{equation}\label{eqjohn1}k_G(w_j, w_{j+1})\leq
\int_{\gamma[w_j,w_{j+1}]}\frac{|dw|}{d_G(w)}<128\ell_{k_D}(\gamma[w_j,w_{j+1}])\leq
256 k_D(w_j, w_{j+1}).\end{equation} By Claim \ref{cl-4}, we also
know that if $p\geq 2$, then for each even number $j\in
\{1,\ldots,p\}$, $w_j\in
\overline{\mathbb{B}}(w_{j-1},\frac{1}{32}d_D(w_{j-1}))$. Hence by
Corollary   \ref{cor4-2} and Claim \ref{cl-4},  we have
\begin{equation}\label{eqjohn2}k_G(w_{j-1},w_j)\leq \mu
\log\left(1+\frac{|w_{j-1}-w_j|}{\min\{d_G(w_{j-1}),d_G(w_j)\}}\right)\leq
128\mu k_D(w_{j-1},w_j),\end{equation} where $\mu$ is the constant
from Corollary   \ref{cor4-2}.

Now we divided the rest part of proof into two cases.

\begin{case} $d_G(y)\geq \frac{1}{128}d_D(y)$. \end{case}

By \eqref{th-eq1} and \eqref{theq00}, we have \begin{equation}\label{eqjohn3'}|x_0-w_0|=\frac{1}{32}d_D(x_0)< \frac{3}{2}d_G(w_0)\leq 3|y-w_0|\end{equation} and \begin{equation}\label{eqjohn3''}|x_0-y|\leq |x_0-w_0|+|y-w_0|\leq 4|y-w_0|,\end{equation}
which, together with $\eqref{th-eq1}$, \eqref{th-eq1'}, Claim \ref{cl-4}, \eqref{eqjohn1} and \eqref{eqjohn2}, shows that
\begin{eqnarray}\label{eqjohn3}k_G(w_0,y)&\leq& \sum^{p-1}_{i=0}k_G(w_i,w_{i+1})
\leq 256\mu \sum^{p-1}_{i=0}k_D(w_i,w_{i+1})\\ \nonumber&\leq&
512\mu k_D(w_0,y)\leq 512\mu( k_D(x_0,w_0)+k_D(x_0,y))\\ \nonumber
&\leq&
512\mu\psi\left(\frac{|x_0-w_0|}{\min\{d_D(x_0),d_D(w_0)\}}\right)\\\nonumber
&+&512\mu\psi\left(\frac{|x_0-y|}{\min\{d_D(x_0),d_D(y)\}}\right)\\
\nonumber&\leq&
\psi_2\left(\frac{|y-w_0|}{\min\{d_G(y),d_G(w_0)\}}\right),\end{eqnarray}
where $\psi_2(t)=1024\mu\psi(8t)$.

\begin{case} $d_G(y)< \frac{1}{128}d_D(y)$. \end{case} In this case, by Claim
\ref{cl-4}, we see that $p$ must be an even number and $p\geq 2$,
and then $y\in \overline{\mathbb{B}}(w_{p-1},
\frac{1}{32}d_D(w_{p-1}))$. If $w_0\in
\overline{\mathbb{B}}(w_{p-1}, \frac{1}{32}d_D(w_{p-1})),$ then by
Corollary  \ref{cor4-2} and Claim \ref{cl-4},  we get
$$k_G(w_{0},y)\leq \mu
\log\left(1+\frac{|w_{0}-y|}{\min\{d_G(w_{0}),d_G(y)\}}\right)\leq
\psi_3\left(\frac{|w_{0}-y|}{\min\{d_G(w_{0}),d_G(y)\}}\right),$$where
$\psi_3(t)=\mu\psi(t)$.

If $w_0\notin \overline{\mathbb{B}}(w_{p-1}, \frac{1}{32}d_D(w_{p-1})),$ then by \eqref{eqjohn3''},  $$d_G(y)<\frac{1}{128}d_D(y)\leq \frac{1}{128}(d_D(w_{p-1})+|w_{p-1}-y|)<\frac{1}{64}d_D(w_{p-1})$$
 and $$|w_{p-1}-x_0|\leq |w_{p-1}-y|+|x_0-y|< 5|y-w_0|,$$which, together with Lemma \ref{lem4-0}, \eqref{th-eq1}, \eqref{eqjohn1}, \eqref{eqjohn2} and \eqref{eqjohn3'}, shows that
\begin{eqnarray*}k_G(w_0,y)&\leq& 256\mu \sum^{p-2}_{i=0}k_D(w_i,w_{i+1})+k_G(w_{p-1},y)\\\nonumber &\leq&
512\mu
k_D(w_0,w_{p-1})+\mu\log\left(1+\frac{|w_{p-1}-y|}{\min\{d_G(w_{p-1}),d_G(y)\}}\right)\\\nonumber
&\leq&512\mu(k_D(x_0,w_0)+k_D(x_0,w_{p-1}))+\mu\log\left(1+\frac{|w_{p-1}-y|}{\min\{d_G(w_{p-1}),d_G(y)\}}\right)\\\nonumber
&\leq&512\mu
\left(\psi\left(\frac{|x_0-w_0|}{\min\{d_D(x_0),d_D(w_0)\}}\right)+\psi\left(\frac{|x_0-w_{p-1}|}{\min\{d_D(x_0),d_D(w_{p-1})\}}\right)\right)
\\\nonumber&+&\mu\log\left(1+\frac{|w_{p-1}-y|}{\min\{d_G(w_{p-1}),d_G(y)\}}\right)\\\nonumber &\leq&
\psi_1\left(\frac{|w_0-y|}{\min\{d_G(w_0),d_G(y)\}}\right),
\end{eqnarray*} where $\psi_4(t)=1025\mu\psi(8t).$
Hence \eqref{eqjohn} holds with $\psi'(t)= 1025\mu\psi(8t)$.

%

\medskip
Now we are going to prove the sufficient part of Theorem \ref{thm3}.

Assume that $G$ is $\psi_1$-John domain with center $z_0$, where $z_0\in G$.
By Lemma \ref{lem4-2}, there exists some point $y_0$ in $ \mathbb{S}(z_0,\frac{1}{32}d_D(z_0))$ such that \begin{equation}\label{th-new-eq1}\frac{1}{48}d_D(z_0)<\frac{1}{33}d_D(y_0)\leq d_{G}(y_0)\leq \frac{33}{31}d_D(y_0).\end{equation}
 We show that there exists a homeomorphism $\psi$ of $[0,\infty)$ such that $D$ is a $\psi$-John domain with center $y_0$.   By the  necessary part of the theorem, we know that $G_1=G\setminus\{z_0\}$ is a $\psi'$-John domain with center $y_0$, where $\psi'(t)= 1025\mu\psi_1(8t)$.

For $y\in D$, if $|y_0-y|\leq \frac{1}{2}\max\{d_D(y_0),d_D(y)\}$, then Lemmas \ref{lem4-0} and Corollary \ref{cor4-1} show that
$$k_D(y_0,y)\leq 2 \log\left(1+\frac{|y_0-y|}{\min\{d_D(y_0),d_D(y)\}}\right)\leq \psi'_1\left(\frac{|y_0-y|}{\min\{d_D(y_0),d_D(y)\}}\right),$$ where $\psi'_1(t)=2\psi_1(t)$ and $2$ is from Corollary \ref{cor4-1}.
In the following, we assume that \begin{equation}\label{eqjohn7}|y_0-y|\geq \frac{1}{2}\max\{d_D(y),d_D(y_0)\}.\end{equation}
If $d_D(y)\leq 62 d_G(y)$, then by \eqref{th-new-eq1},$$|y-z_0|\geq |y-y_0|-|y_0-z_0|\geq \frac{1}{128}d_D(z_0).$$
Now we claim that \begin{equation}\label{eqjohn6}d_G(y)\leq 129d_{G_1}(y).\end{equation} In fact, if $d_G(y)=d_{G_1}(y)$, then the above inequality is obvious. If $d_G(y)> d_{G_1}(y)$, then $d_{G_1}(y)=|y-z_0|$. Hence $$d_G(y)\leq d_D(z_0)+|z_0-y|\leq 129d_{G_1}(y),$$ which shows \eqref{eqjohn6}.

Similarly, we have \begin{equation}\label{eqjohn8}d_G(y_0)\leq 129 d_{G_1}(y_0).\end{equation}  Hence \eqref{th-new-eq1} and \eqref{eqjohn6} yields

 $$k_D(y_0,y)\leq k_{G_1}(y_0,y)\leq \psi'\left(\frac{|y_0-y|}{\min\{d_{G_1}(y_0),d_{G_1}(y)\}}\right)\leq \psi'_2\left(\frac{|y_0-y|}{\min\{d_{D}(y_0),d_{D}(y)\}}\right),$$ where $\psi'_2(t)= 1025\mu\psi_1(2^{15}t).$

 If $d_D(y)\geq 62 d_G(y)$, then for  $y_1\in \mathbb{S}(y,\frac{1}{16}d_D(y))$,  Lemma \ref{star} implies \begin{equation}\label{eqjohn9}d_{D}(y_1)\leq d_D(y)+|y_1-y|\leq 32 d_G(y_1).\end{equation}
Hence, a similar proof as to \eqref{eqjohn6} leads to $$d_G(y_1)\leq
129 d_{G_1}(y_1),$$  which, together with Corollary \ref{cor4-1},
\eqref{th-new-eq1}, \eqref{eqjohn7}, \eqref{eqjohn8} and
\eqref{eqjohn9}, shows that \begin{eqnarray*}k_D(y_0,y)&\leq&
k_{G_1}(y_0,y_1)+ k_D(y_1,y)\\ &\leq&
\psi'\left(\frac{|y_0-y_1|}{\min\{d_{G_1}(y_0),d_{G_1}(y_1)\}}\right)+2\log\left(1+\frac{|y_1-y|}{\min\{d_D(y_1),d_D(y)\}}\right)\\
&\leq&
\psi'_3\left(\frac{|y_0-y|}{\min\{d_D(y_0),d_D(y)\}}\right),\end{eqnarray*}
where $\psi'_3(t)= 1025(\mu+2)\psi_1(2^{15}t)$, and $\mu$ is the
constant from Corollary \ref{cor4-2}.  By letting
$\psi(t)=1025(\mu+2)\psi_1(2^{15}t),$ we get the sufficient part of
the theorem. Hence the proof of the theorem is complete.
\end{proof}

\begin{remark} Let $\psi:[0,\infty]\to[0,\infty]$ be a homeomorphism and $c$, $\lambda_1$, $\lambda_2$ be positive constants. We define the following class:
$$\Psi_{\lambda_1,\lambda_2}=\{\psi: \lambda_1\leq \frac{\psi(ct)}{\psi(t)}\leq \lambda_2\}.$$

The proof of
Theorem \ref{thm3} yields the following
quantitative statement: $\psi_1(t) = b_1 \psi(b_2 t)$ and  $\psi_2(t) = b_3 \psi(b_4 t)$ for some positive constants $b_j$.
Thus we see that if $D$ is a $\psi$-John domain with $\psi\in \Psi_{\lambda_1,\lambda_2}$, then $D\setminus P_D$ is a $\psi_1$-John domain with $\psi_1\in \Psi_{\lambda_1,\lambda_2}$. The converse implication also holds.
\end{remark}

\end{document}